\documentclass[11pt,reqno]{amsart}
\usepackage{amssymb}
\usepackage{mathtools}
\usepackage{mathabx}
\usepackage{mathrsfs}
\usepackage{verbatim}  
\usepackage{wrapfig}
\usepackage{xcolor}
\usepackage{hyperref}
\usepackage{cancel}
\usepackage[margin=1.3in]{geometry}

\hypersetup{
	colorlinks=true,
	linkcolor=blue,
	filecolor=magenta,      
	urlcolor=cyan,
	citecolor={green!70!black},
}

\newtheorem{thm}{Theorem}[section]
\newtheorem{lem}[thm]{Lemma}

\newtheorem{cor}[thm]{Corollary}

\newtheorem{Def}[thm]{Definition}

\newtheorem{prop}[thm]{Proposition}
\newtheorem{rem}[thm]{Remark}
\newtheorem{ex}[thm]{Example}

\newcommand{\bdfn}{\begin{Def} \rm}
	\newcommand{\edfn}{\end{Def}}

\newcommand{\es}{\emptyset}
\newcommand{\ci}{\subseteq}

\newcommand{\al}{\alpha}
\newcommand{\be}{\beta}
\newcommand{\de}{\delta}
\newcommand{\e}{\varepsilon}

\newcommand{\la}{\lambda}

\newcommand{\ga}{\gamma}

\newcommand{\trcp}{\textrm{r.c.p.}}
\newcommand{\ercp}{\emph{r.c.p.}}

\newcommand{\mb}{\mathbb}
\newcommand{\mc}{\mathcal}

\newcommand{\iy}{\infty}
\newcommand{\msc}{\mathscr}

\newcommand{\beqa}{\begin{eqnarray*}}
	\newcommand{\eeqa}{\end{eqnarray*}}

\newcounter{cnt1}
\newcounter{cnt2}
\newcounter{cnt3}
\newcounter{cnt4}
\newcommand{\blr}{\begin{list}{$($\roman{cnt1}$)$} {\usecounter{cnt1}
			\setlength{\topsep}{0pt} \setlength{\itemsep}{0pt}}}
	\newcommand{\blR}{\begin{list}{\Roman{cnt4}.\ } {\usecounter{cnt4}
				\setlength{\topsep}{0pt} \setlength{\itemsep}{0pt}}}
		\newcommand{\bla}{\begin{list}{$(\alph{cnt2})$} {\usecounter{cnt2}
					\setlength{\topsep}{0pt} \setlength{\itemsep}{0pt}}}
			\newcommand{\bln}{\begin{list}{$($\arabic{cnt3}$)$} {\usecounter{cnt3}
						\setlength{\topsep}{0pt} \setlength{\itemsep}{0pt}}}
				\newcommand{\el}{\end{list}}

			\title{On Property-$(P_{1})$ in Banach Spaces}
			\author[Thomas]{Teena Thomas}
			\address[Teena Thomas]{Department of Mathematics \\
				Indian Institute of Technology Hyderabad \\
				India, \textit{E-mail~:} \textit{ma19resch11003@iith.ac.in}/\textit{tteena.tthomas@gmail.com}}
			
			
\begin{document}
				\maketitle
				\begin{abstract}
					We discuss a set-valued generalization of strong proximinality in Banach spaces, introduced by J. Mach [Continuity properties of Chebyshev centers. \textit{J. Approx. Theory}, 29(3):223–230, 1980] as property-$(P_1)$. For a Banach space $X$, a closed convex subset $V$ of $X$ and a subclass $\msc{F}$ of the closed bounded subsets of $X$, this property, defined for the triplet $(X,V,\msc{F})$, describes simultaneous strong proximinality of $V$ at each of the sets in $\msc{F}$. We establish that if the closed unit ball of a closed subspace of a Banach space $X$ possesses property-$(P_1)$ for each of the classes of closed bounded, compact and finite subsets of $X$, then so does the subspace. It is also proved that the closed unit ball of an $M$-ideal in an $L_{1}$-predual space satisfies property-$(P_{1})$ for the compact subsets of the space. For a Choquet simplex~$K$, we provide a sufficient condition for the closed unit ball of a finite co-dimensional closed subspace of $A(K)$ to satisfy property-$(P_{1})$ for the compact subsets of $A(K)$. This condition also helps to establish the equivalence of strong proximinality of the closed unit ball of a finite co-dimensional subspace of $A(K)$ and property-$(P_1)$ of the closed unit ball of the subspace for the compact subsets of $A(K)$. Further, for a compact Hausdorff space~$S$, a characterization is provided for a strongly proximinal finite co-dimensional closed subspace of $C(S)$ in terms of property-$(P_{1})$ of the subspace and that of its closed unit ball for the compact subsets of $C(S)$. We generalize this characterization for a strongly proximinal finite co-dimensional closed subspace of an $L_{1}$-predual space. As a consequence, we prove that such a subspace is a finite intersection of hyperplanes such that the closed unit ball of each of these hyperplanes satisfy property-$(P_1)$ for the compact subsets of the $L_1$-predual space and vice versa. We conclude this article by providing an example of a closed subspace of a non-reflexive Banach space which satisfies $1 \frac{1}{2}$-ball property and does not admit restricted Chebyshev center for a closed bounded subset of the Banach space.
				\end{abstract}
				
				
				\section{Introduction}\label{sec1}
				The concepts of best simultaneous approximation and in particular, proximinality in Banach spaces are of great interest and significance in approximation theory. The classical (restricted) Chebyshev center problem stems from these concepts. With its first appearance in \cite{GV}, the notion of strong proximinality rose to prominence, which is evident through \cite{BLR}, \cite{SD2}, \cite{SD} and \cite{C}--\cite{CT}. This article aims to explore its generalization, introduced as property-$(P_1)$ in \cite{M1}, in certain objects of the class of Banach spaces.
				
				In this article, we consider Banach spaces only over the real field $\mb{R}$ and all the subspaces considered are assumed to be closed. Let $X$ be a Banach space. For $x \in X$ and $r>0$, $B[x,r]$ denotes the closed ball centered at $x$ with radius $r$. In particular, for simplicity, we denote the closed unit ball $B[0,1]$ by $B_{X}$. The dual space of $X$ is denoted by $X^{\ast}$. If $Y$ is a subspace of $X$, then $B_{Y} = B_{X} \cap Y$. For a non-empty closed convex subset $V$ of $X$, let $\mc{CB}(V)$, $\mc{K}(V)$ and $\mc{F}(V)$ denote the classes of all non-empty closed bounded subsets of $V$, non-empty compact subsets of $V$ and non-empty finite subsets of $V$ respectively.
				
				Let $B \in \mc{CB}(X)$ and $V$ be a non-empty closed convex subset of $X$. For each $x \in X$, let $r(x,B) = \sup\{\|x-b\|:b \in B\}$ and for each $\la>0$, let $S_{\la}(B) = \{x \in X: r(x,B) \leq \la\}$. The \textit{restricted Chebyshev radius} of $B$ with respect to(in short, w.r.t.) $V$ in $X$ is denoted by $\emph{rad}_{V}(B)$ and is defined as $\emph{rad}_{V}(B) = \inf_{v \in V} r(v,B)$. A point~$v \in X$ is called a \textit{restricted  Chebyshev center} of $B$ w.r.t. $V$ in $X$ if $v \in S_{\emph{rad}_{V}(B)}(B) \cap V$. We denote the set of all restricted Chebyshev centers of $B$ w.r.t. $V$ in $X$ by $\emph{cent}_{V}(B)$. For $\de>0$, we define $\emph{cent}_{V}(B,\de)= \{v \in V: r(v,B) \leq \emph{rad}_{V}(B)+\de\}$. Let us note here that $\emph{cent}_{V}(B,\de) = S_{\emph{rad}_{V}(B)+\de}(B) \cap V.$ If $V = X$, then $\emph{rad}_{X}(B)$ is called the \textit{Chebyshev radius} of $B$ in $X$ and the elements in $\emph{cent}_{X}(B)$ are called the \textit{Chebyshev centers} of $B$ in $X$. 
				
				\begin{Def}[{\cite{PN}}]
					Let $V$ be a non-empty closed convex subset of $X$ and $\msc{F} \ci \mc{CB}(X)$. Then the pair $(V,\msc{F})$ is said to satisfy the \textit{restricted center property}(in short, \trcp) if for each $F \in \msc{F}$, $\textrm{cent}_{V}(F) \neq \es$.
				\end{Def}
				
				A non-empty closed convex subset~$V$ of $X$ is said to be proximinal in $X$ if for each $x \in X$, $\emph{cent}_{V}(\{x\}) \neq \es$. For each $x \in X$, we denote $\emph{cent}_{V}(\{x\})$ by $P_{V}(x)$ and $\emph{rad}_{V}(\{x\}) = \inf_{v \in V} \|x-v\|$ is the distance of the point~$x$ from $V$, which we denote by $d(x,V)$. We say a subspace~$Y$ of $X$ is ball proximinal in $X$ if $B_{Y}$ is proximinal in $X$.
				
				The following definition is a stronger form of proximinality, which was introduced in \cite{GV}.
				\begin{Def}\label{Def0.0}
					A proximinal subset~$V$ of a Banach space~$X$ is said to be strongly proximinal at $x \in X$ if for each $\e>0$, there exists $\de(\e,x)>0$ such that $P_{V}(x,\de) \ci P_{V}(x)  +\e B_{X}$, where $P_{V}(x,\de) = \textrm{cent}_{V}(\{x\},\de)$. 
					We say that $V$ is strongly proximinal in $X$ if it is strongly proximinal at all points in $X$. 
					
					A subspace~$Y$ of a Banach space~$X$ is said to be strongly ball proximinal in $X$ if $B_Y$ is strongly proximinal in $X$.
				\end{Def} 
				
				The set-valued analogue of strong proximinality was first introduced by J. Mach in \cite{M1} and is defined as follows. 
				
				\begin{Def}\label{Def0}
					Let $X$ be a Banach space, $V$ be a non-empty closed convex subset of $X$ and $\msc{F} \ci \mc{CB}(X)$ such that $(V,\msc{F})$ has \trcp. Then the triplet $(X,V,\msc{F})$ has property-$(P_{1})$ if for each $\e>0$ and $F \in \msc{F}$, there exists $\de(\e,F)>0$ such that $\textrm{cent}_{V}(F,\de) \ci \textrm{cent}_{V}(F) + \e B_{X}.$
				\end{Def}
				
				It is clear from the definition of property-$(P_1)$ that if $V$ is a subspace and $\msc{F}$ is the class of all singleton subsets of $X$, then $V$ is strongly proximinal in $X$ if $(X,V,\msc{F})$ has property-$(P_1)$. Now, with the same notations as in Definition~\ref{Def0}, an equivalent way of saying that the triplet $(X,V,\msc{F})$ satisfies property-$(P_1)$ is if the sequence~$\{v_{n}\} \ci V$ is such that $r(v_{n},F) \rightarrow \emph{rad}_{V}(F)$, then $d(v_{n},\emph{cent}_{V}(F)) \rightarrow 0$. Examples of triplets satisfying property-$(P_{1})$ can be found in \cite{M1}.
				
				It is proved in \cite{BLR} that the notion of strong ball proximinality is stronger than that of strong proximinality for a subspace of a Banach space. In Section~\ref{sec2}, we explore such a connection between the notion of property-$(P_1)$ of a subspace of a Banach space~$X$ and that of its closed unit ball for $\mc{CB}(X)$, $\mc{K}(X)$ and $\mc{F}(X)$. 
				In Sections~\ref{sec3} and $\ref{sec4}$, we mainly investigate property-$(P_1)$ in the class of $L_{1}$-predual spaces. Let us recall some of the basic notions and well-known results in an $L_{1}$-predual space.
				\begin{Def}
					A Banach space X is said to be an $L_{1}$-predual space if $X^{\ast}$ is isometric to an $L_{1}(\mu)$ space, where $(\varOmega,\varSigma,\mu)$ is a positive measure space.
				\end{Def}
				
				J. Lindenstrauss characterized $L_{1}$-predual spaces in terms of the intersection properties of the balls in these spaces in \cite{JL}. 
				\begin{Def}
					Let $X$ be a Banach space and $n \in \mb{N}$. Then $X$ is said to have the $n.2.I.P.$ if for every family of pairwise intersecting balls $\{B[x_{i},r_{i}]:i=1,\ldots,n\}$, $\bigcap_{i=1}^{n} B[x_{i},r_{i}] \neq \es$.
				\end{Def}
				A detailed study on these intersection properties can be found in \cite{JL}. It is proved in \cite[Theorem~6.1]{JL} that $X$ is an $L_{1}$-predual space if and only if for each $n \in \mb{N}$, $X$ has $n.2.I.P.$. The class of spaces of real-valued continuous functions on a compact Hausdorff space~$S$ equipped with supremum norm, denoted by $C(S)$ and that of spaces of real-valued affine continuous functions on a Choquet simplex~$K$ equipped with supremum norm, denoted by $A(K)$, are two major subclasses of the $L_{1}$-predual spaces (see \cite{AE} and \cite{JL}). 
				We refer \cite{AL} and \cite{AE} for a detailed study on Choquet simplex and Choquet theory in general. For a closed convex set~$V$, the set of all extreme points of $V$ is denoted by $ext(V)$. If $\mu$ is a regular Borel measure on a compact Hausdorff space~$S$, then the support of $\mu$ is denoted by $S(\mu)$. 
				
				The differentiability notion, introduced in \cite{FP} as strongly subdifferentiable(in short, SSD) points, characterizes strongly proximinal hyperplanes. In \cite{GV}, it is proved that for a Banach space~$X$ and $x^\ast \in X^{\ast}$, $x^\ast$ is an SSD-point of $X^{\ast}$ if and only if the kernel of $x^\ast$, denoted by $ker(x^\ast)$, is strongly proximinal in $X$. It is also established that if $Y$ is a strongly proximinal finite co-dimensional subspace of a Banach space, then the annihilator of $Y$, denoted by $Y^{\perp}$, is contained in the set of all SSD-points of $X^{\ast}$. If $X$ is an $L_{1}$-predual space, then the converse is also true(see \cite[Proposition~3.20]{CT}).
				
				Let us now recall another notion in a Banach space, which is stronger than proximinality, called as an $M$-ideal. A detailed study of $M$-ideals can be found in \cite{HW}.
				\begin{Def}
					Let $X$ be a Banach space. 
					\begin{enumerate}
						\item A linear projection~$P$ on $X$ is said to be an $L$-projection if $\|x\|= \|Px\| +\|x-Px\|$, for each $x \in X$. 
						\item A subspace~$J$ of a Banach space~$X$ is said to be an $L$-summand in $X$ if it is the range of an $L$-projection.
						\item A subspace~$J$ of a Banach space~$X$ is said to be an $M$-ideal in $X$ if $J^{\perp}$ is an $L$-summand.
					\end{enumerate}
				\end{Def}
				
				Another subclass of the $L_{1}$-predual spaces is the class of $M$-ideals in an $L_{1}$-predual space. It is proved in \cite{M1} that if $J$ is an $M$-ideal in an $L_{1}$-predual space $X$, then $(X,J,\mc{K}(X))$ satisfies property-$(P_{1})$. This motivates us to investigate if the triplet $(X,B_{J},\mc{K}(X))$ also satisfies property-$(P_1)$ or not. The answer is in the affirmative and is proved in Section~\ref{sec3}.
				For a Choquet simplex~$K$, if $\mu \in A(K)^{\ast}$, then it means $\mu \in C(K)^{\ast}$ is a restriction map on $A(K)$. Now, if $Y$ is a finite co-dimensional subspace of $A(K)$, then we prove in Theorem~3.7 that $(A(K),B_Y,\mc{K}(A(K)))$ satisfies property-$(P_1)$ by imposing the conditions that the support of each of the defining measures of the subspace is finite and is contained in $ext(K)$. As a consequence, in particular, the condition that the support of each of the defining measures of the subspace is contained in $ext(K)$ also establishes that the closed unit ball $B_Y$ being strongly proximinal in $A(K)$ is equivalent to the triplet $(A(K),B_Y,\mc{K}(A(K)))$ satisfying property-$(P_1)$. Further, in Sections~\ref{sec3}, we also prove that for a compact Hausdorff space~$S$, $Y$ is a strongly proximinal finite co-dimensional subspace of $C(S)$ if and only if the triplet $(C(S), B_{Y},\mc{K}(C(S)))$ satisfies property-$(P_1)$. The equivalence of the triplets $(C(S),Y,\mc{K}(C(S)))$ and $(C(S), B_{Y},\mc{K}(C(S)))$ satisfying property-$(P_1)$ is also established. These results are then generalized in Section~\ref{sec4} for a strongly proximinal finite co-dimensional subspace of an $L_{1}$-predual space, thereby adding two more characterizations to the list in \cite[Theorem~2.6]{C}.
				
				We now recall the notion of $1 \frac{1}{2}$-ball property, which was first introduced in \cite{Y}. 
				\begin{Def}
					A subspace~$Y$ of a Banach space~$X$ is said to have $1 \frac{1}{2}$-ball property in $X$ if for each $y \in Y$, $x \in X$ and $r_1,r_2>0$, if $\|x-y\| < r_1 + r_2$ and $Y \cap B[x,r_2] \neq \es$, then $Y \cap B[y,r_1] \cap B[x,r_2] \neq \es$.
				\end{Def} 
				It is proved in \cite[Proposition~3.3]{SD2} that if $Y$ satisfies $1 \frac{1}{2}$-ball property in a Banach space $X$, then $Y$ is strongly proximinal in $X$. In Section~\ref{sec5}, we provide an example of a hyperplane in a non-reflexive Banach space $X$ which satisfies $1 \frac{1}{2}$-ball property and does not satisfy \ercp~for $\mc{F}(X)$.
				
				\section{Property-$(P_{1})$ of a Banach space in relation to that of its closed unit ball}\label{sec2}
				In this section, for a subspace $Y$ of a Banach space $X$, we prove that if $\msc{F} =\mc{CB}(X)$, $\mc{K}(X)$ or $\mc{F}(X)$ such that $(X,B_{Y},\msc{F})$ has property-$(P_{1})$ then so does $(X,Y,\msc{F})$. The ideas used are similar to the ones in \cite{BLR}.
				If $\la>0$, then for a non-empty set~$A \ci X$, the set~$\{\la a : a \in A\}$ is denoted by $\la A$. 
				\begin{lem}\label{lem0.1}
					Let $Y$ be a subspace of a Banach space~$X$ and $B \in \mc{CB}(X)$.
					\blr
					\item For each $\la>0$, $\la \textrm{cent}_{B_{Y}}(\frac{1}{\la}B) = \textrm{cent}_{\la B_{Y}}(B)$.
					\item For each $\la \geq \sup_{b \in B} \|b\|  + \textrm{rad}_{Y}(B)$, $\textrm{cent}_{Y}(B) \ci \textrm{cent}_{\la B_{Y}}(B)$.
					\item For each $\la > \sup_{b \in B} \|b\|  + \textrm{rad}_{Y}(B)$, $\textrm{cent}_{Y}(B) = \textrm{cent}_{\la B_{Y}}(B)$.
					\el
				\end{lem}
				\begin{proof}
					$(i)$. Let $\la>0$ and $y_{0} \in B_{Y}$. $\la y_{0} \in \la \emph{cent}_{B_{Y}}(\frac{1}{\la}B)$ $\Leftrightarrow$ for each $y \in B_{Y}$, $r(y_{0},\frac{1}{\la}B) \leq r(y,\frac{1}{\la}B)$ $\Leftrightarrow$ for each $y \in B_{Y}$, $r(\la y_{0}, B) \leq r(\la y, B)$ $\Leftrightarrow$ $\la y_{0} \in \emph{cent}_{\la B_{Y}}(B)$.
					
					$(ii)$. Let $\la \geq \sup_{b \in B} \|b\|  + \emph{rad}_{Y}(B)$ and $y_{0} \in \emph{cent}_{Y}(B)$. Then for each $b \in B$, $\|y_{0}\| \leq \|b\| + \|y_{0}-b\| \leq \sup_{b \in B} \|b\| + r(y_{0},B) = \sup_{b \in B} \|b\|  + \emph{rad}_{Y}(B) \leq \la$. Hence, $y_{0} \in \la B_{Y}$ and it follows that $y_{0} \in \emph{cent}_{\la B_{Y}}(B)$.
					
					$(iii)$. Let $\la > \sup_{b \in B} \|b\|  + \emph{rad}_{Y}(B)$ and $y_{0} \in \emph{cent}_{\la B_{Y}}(B)$. Let $R= \emph{rad}_{Y}(B)$. It is easy to see that for each $\de>0$, $R = \inf \{r(y,B): y \in S_{R+\de}(B) \cap Y\}$. In particular, let $\de = \la - (\sup_{b \in B} \|b\|  +R)$. If $y \in S_{R+\de}(B) \cap Y$, then $y \in \la B_{Y}$. Hence, $r(y_{0},B) \leq r(y,B)$. It follows that $y_{0} \in \emph{cent}_{Y}(B)$. 
				\end{proof}
				
				\begin{prop}\label{prop0.1}
					Let $Y$ be a subspace of a Banach space~$X$ and $\msc{F} = \mc{CB}(X)$, $\mc{K}(X)$ or $\mc{F}(X)$. If $(B_{Y},\msc{F})$ has \trcp, then so does $(Y,\msc{F})$.
				\end{prop}
				\begin{proof}
					We prove the result only for $\mc{CB}(X)$ because the same proof works for $\mc{K}(X)$ and $\mc{F}(X)$. Let $B \in \mc{CB}(X)$ and $\la> \sup_{b \in B} \|b\| + \emph{rad}_{Y}(B)$. Since $(B_{Y},\mc{CB}(X))$ has \ercp, $(B_{Y},\mc{CB}(B_X))$ has \ercp. Therefore, $(\la B_{Y},\mc{CB}(\la B_{X}))$ has \ercp. Clearly, for each $b \in B$, $b \in \la B_{X}$. Therefore, from Lemma~\ref{lem0.1} $(iii)$, $\emph{cent}_{Y}(B) = \emph{cent}_{\la B_{Y}}(B) \neq \es$. 
				\end{proof}
				
				\begin{prop}\label{prop0.2}
					Let $Y$ be a subspace of a Banach space~$X$ and $B \in \mc{CB}(X)$. Then
					\blr 
					\item For each $\la>0$ and $\de>0$, $\textrm{cent}_{\la B_{Y}}(B,\de) = \la \textrm{cent}_{B_{Y}}(\frac{1}{\la}B, \frac{\de}{\la})$.
					\item For each $\la>0$, $(X, \la B_{Y}, \{B\})$ has property-$(P_1)$ if and only if $(X, B_{Y}, \{\frac{1}{\la} B\})$ has property-$(P_{1})$.
					\item Let $\msc{F}  =\mc{CB}(X)$, $\mc{K}(X)$ or $\mc{F}(X)$. If $(X,B_{Y},\msc{F})$ has property-$(P_1)$, then so does $(X,Y,\msc{F})$.
					\el
				\end{prop}	
				\begin{proof}
					$(i)$ follows from a similar argument as in Lemma~\ref{lem0.1} $(i)$.
					
					$(ii)$ easily follows from $(i)$.
					
					$(iii)$. We prove the result only for $\mc{CB}(X)$ because the same proof works for $\mc{K}(X)$ and $\mc{F}(X)$. Assume $(X,B_{Y},\mc{CB}(X))$ has property-$(P_1)$. Obviously, $(X,B_{Y},\mc{CB}(B_X))$ has property-$(P_1)$ and from Proposition~\ref{prop0.1}, it follows that $(Y,\mc{CB}(X))$ has \ercp. Let $B \in \mc{CB}(X)$ and $\la > \sup_{b \in B} \|b\|  + \emph{rad}_{Y}(B)$. Therefore, $(X,B_{Y},\{\frac{1}{\la} B\})$ has property-$(P_1)$ and hence, from $(ii)$, $(X,\la B_{Y}, \{B\})$ has property-$(P_1)$. Now, using the same argument as in Lemma~\ref{lem0.1} $(iii)$, for $0<\de< \la -(\sup_{b \in B} \|b\|  + \emph{rad}_{Y}(B))$, $\emph{cent}_{Y}(B,\de) \ci \la B_{Y}$ and hence, $\emph{cent}_{Y}(B,\de) = \emph{cent}_{\la B_{Y}}(B,\de)$. Also, $\emph{cent}_{Y}(B) = \emph{cent}_{\la B_Y}(B)$. It follows that $(X,Y,\{B\})$ has property-$(P_{1})$. Therefore, $(X,Y,\mc{CB}(X))$ has property-$(P_1)$.
				\end{proof}	 
				
				\section{Property-$(P_1)$ in some $L_{1}$-predual spaces}\label{sec3}
				In this section, we study property-$(P_1)$ in few important subclasses of the class of $L_{1}$-predual spaces.
				
				We first aim to show that if $J$ is an $M$-ideal in an $L_{1}$-predual space $X$, then the triplet $(X,B_{J},\mc{K}(X))$ satisfies property-$(P_{1})$. The following lemma is obtained by minor modifications to the proof of \cite[Lemma~2.1]{M2}.
				
				\begin{lem}\label{lem0.0}
					Let $X$ be an $L_{1}$-predual space, $J$ be an $M$-ideal in $X$, $F \in \mc{K}(X)$, $\{x_{1},\ldots,x_{n}\} \ci X$ and $r,r_{1},\ldots,r_{n}>0$. If for each $x \in F$, $B[x,r] \cap J \neq \es$; for each $i=1,\dots,n$, $B[x_{i},r_{i}]\cap J \neq \es$ and $\bigcap_{i=1}^{n} B[x_{i},r_{i}] \cap S_{r}(F) \neq \es$, then $\bigcap_{i=1}^{n} B[x_{i},r_{i}] \cap S_{r}(F) \cap J \neq \es.$
				\end{lem}
				
				\begin{thm}\label{thm0}
					Let $X$ be an $L_{1}$-predual space and $J$ be an $M$-ideal in $X$. Then $(X, B_{J}, \mc{K}(X))$ has property-$(P_{1})$.
				\end{thm}
				\begin{proof}
					Let $\e>0$ and $F \in \mc{K}(X)$. Let $x \in \emph{cent}_{B_{J}}(F,\e) = S_{\emph{rad}_{B_{J}}(F)+\e}(F) \cap B_{J}$. Obviously, $B[x,\e] \cap B_{X}\neq \es$ and for each $y \in F$, $B[x,\e] \cap B[y,\emph{rad}_{B_{J}}(F)] \neq \es$. By \cite[Corollary~4.8]{CT}, $J$ is ball proximinal in $X$. Hence for each $y \in F$, $B[y,d(y,B_{J})] \cap B_{J} \neq \es$. Also, clearly, for each $y \in F$, $d(y,B_{J}) \leq \emph{rad}_{B_{J}}(F)$. It follows that for each $y \in F$, $B[y,\emph{rad}_{B_{J}}(F)] \cap B_{X} \neq \es$. By \cite[Theorem~2.2]{M2}, $\emph{cent}_{J}(F) = S_{\emph{rad}_{J}(F)}(F) \cap J \neq \es$. Since $\emph{rad}_{J}(F) \leq \emph{rad}_{B_{J}}(F)$, $S_{\emph{rad}_{B_{J}}(F)}(F) \neq \es.$ Now, $\left\{B[y,\emph{rad}_{B_{J}}(F)]:y \in F\right\} \cup \left\{B[x,\e],B_{X}\right\}$ is a collection of closed balls which intersect pairwise. Therefore, by \cite[Theorem~4.5, pg.~38]{JL} and \cite[Theorem~6, pg.~212]{L}, \[B[x,\e] \cap S_{\emph{rad}_{B_{J}}(F)}(F) \cap B_{X} \neq \es.\] It is easily observed that each of the above closed balls intersects $J$. Therefore, by Lemma~\ref{lem0.0}, \[B[x,\e] \cap S_{\emph{rad}_{B_{J}}(F)}(F) \cap B_{X} \cap J \neq \es.\]
				\end{proof}
				
				For a compact Hausdorff space~$S$, the next main result in this section provides a characterization for a strongly proximinal finite co-dimensional subspace $Y$ of $C(S)$ in terms of property-$(P_1)$ of the triplets $(C(S), Y,\mc{K}(C(S)))$ and $(C(S), B_Y,\mc{K}(C(S)))$. To this end, we need the following lemma, which also aids in proving other results in this article.
				
				\begin{lem}\label{lem1}
					Let $V$ be a non-empty closed convex subset of a Banach space~$X$ and $B \in \mc{CB}(X)$. Then for every $\e>0$ and $\gamma>0$, there exists $\de>0$ such that \[\textrm{cent}_{V}(B,\gamma+\de) \ci \textrm{cent}_{V}(B,\gamma) + \e B_{X}.\]
				\end{lem}
				\begin{proof}
					Let $\e>0$, $\ga>0$ and $R = \emph{rad}_{V}(B)$. We choose $\de>0$ such that $\de< \min\left\{R,\frac{\e \ga}{6R + 4\ga}\right\}$. Let $v \in \emph{cent}_{V}(B,\ga+\de)$. Then $r(v,B) \leq R+ \ga + \de$. Further, let $v^{\prime} \in \emph{cent}_{V}\left(B, \frac{\ga}{2}\right)$. We define $\la = \frac{2\de}{2 \de + \ga}$ and $\tilde{v} = (1-\la) v + \la v^{\prime}$. After performing some easy calculations, for each $b \in B$, we obtain $\|\tilde{v}-b\|< R+\ga$ and hence, it follows that $r(\tilde{v},B) \leq R+\ga$. Also, for each $b \in B$, \[\|v -\tilde{v}\| \leq \la(\|v-b\| + \|v^{\prime}-b\|) < \frac{2 \de}{2 \de + \ga}\left(3R  + 2 \ga \right) <  \e.\]
				\end{proof}		
				
				\begin{rem}\label{rem1}
					If $X$ is an $L_{1}$-predual space, then it follows from \cite[Corollary~3.4]{BR} and \cite[Theorem~4.5, p.~38]{JL} that for each $F \in \mc{K}(X)$, $\textrm{cent}_{X}(F) \neq \es$.
				\end{rem}
				
				\begin{thm}\label{thm2}
					Let $S$ be a compact Hausdorff space and $\{\mu_{1},\ldots,\mu_{n}\} \ci C(S)^{\ast}$ such that for each $i=1,\ldots,n$, $\|\mu_i\|=1$. If for each $i=1,\ldots,n$, $S(\mu_{i})$ is finite and $Y=\bigcap_{i=1}^{n} ker(\mu_{i})$, then $(C(S),B_{Y},\mc{K}(C(S)))$ has property-$(P_{1})$.
				\end{thm}
				
				\begin{proof}
					We employ techniques similar to those used in the proof of \cite[Proposition~4.2]{CS}. We prove the result only for $n=2$ because the same ideas work to prove the result for $n \neq 2$. Let $\mu_{1} = \sum_{i=1}^{m} \al_{i} \de_{k_{i}}$, $\mu_{2} = \sum_{j=1}^{r} \be_{j} \de_{t_{j}}$, $Y = ker(\mu_{1}) \cap ker(\mu_{2})$ and $F \in \mc{K}(C(S))$.
					
					{\sc Case 1:} $S(\mu_{1}) \cap S(\mu_{2}) =\es$.
					
					Let us define 
					\begin{equation}\label{eqn2.1}
						A = \left\{(\ga_{1},\ldots,\ga_{m},\ga^{\prime}_{1},\ldots,\ga^{\prime}_{r}) \in [-1,1]^{m+r}: \sum_{i=1}^{m} \al_{i} \ga_{i} = 0\mbox{ and }\sum_{j=1}^{r} \be_{j} \ga^{\prime}_{j} = 0\right\}
					\end{equation}
					and 
					\begin{equation}\label{eqn2.2}
						\al = \inf \left\{\sup_{f \in F} \max_{\substack{1 \leq i \leq m\\ 1 \leq j \leq r}} \{|\ga_{i}-f(k_{i})|, |\ga^{\prime}_{j} - f(t_{j})|\}: (\ga_{1},\ldots,\ga_{m},\ga^{\prime}_{1},\ldots,\ga^{\prime}_{r}) \in A\right\}.
					\end{equation}
					
					For each $f \in F$, the continuity of the map \[(\ga_{1},\ldots,\ga_{m},\ga^{\prime}_{1},\ldots,\ga^{\prime}_{r}) \mapsto \max_{\substack{1 \leq i \leq m\\ 1 \leq j \leq r}} \{|\ga_{i}-f(k_{i})|, |\ga^{\prime}_{j} - f(t_{j})|\}\] on $\mb{R}^{m+r}$ implies the lower semicontinuity of the map \[(\ga_{1},\ldots,\ga_{m},\ga^{\prime}_{1},\ldots,\ga^{\prime}_{r}) \mapsto \sup_{f \in F} \max_{\substack{1 \leq i \leq m\\ 1 \leq j \leq r}} \{|\ga_{i}-f(k_{i})|, |\ga^{\prime}_{j} - f(t_{j})|\}\] on $\mb{R}^{m+r}$. The set~$A \ci \mb{R}^{m+r}$ is non-empty and compact and hence, the infimum in ($\ref{eqn2.2}$) is attained. Let $(\eta_{1},\ldots,\eta_{m},\eta^{\prime}_{1},\ldots,\eta^{\prime}_{r}) \in A$ be such that 
					\begin{equation}\label{eqn2.3}
						\al = \sup_{f \in F} \max_{\substack{1 \leq i \leq m\\ 1 \leq j \leq r}} \{|\eta_{i}-f(k_{i})|, |\eta^{\prime}_{j} - f(t_{j})|\}.
					\end{equation}
					Therefore, for each $f \in F$, 
					\begin{equation}\label{eqn2.4}
						\begin{split}
							&-\al+ \eta_{i} \leq f(k_{i}) \leq \al+ \eta_{i}\mbox{ for }i=1,\ldots,m\mbox{ and }\\ &-\al+ \eta^{\prime}_{j} \leq f(t_{j}) \leq \al+ \eta^{\prime}_{j}\mbox{ for }j=1,\ldots,r.\\
						\end{split}
					\end{equation}
					Let $R = \emph{rad}_{B_{Y}}(F)$. It follows from the definition of $\al$ that $R \geq \al$. Therefore, from the inequalities in ($\ref{eqn2.4}$), it follows that for each $f \in F$, 
					\begin{equation}\label{eqn2.6}
						\begin{split}
							&-R+ \eta_{i} \leq f(k_{i}) \leq R+ \eta_{i}\mbox{ for }i=1,\ldots,m\mbox{ and }\\ &-R+ \eta^{\prime}_{j} \leq f(t_{j}) \leq R+ \eta^{\prime}_{j}\mbox{ for }j=1,\ldots,r.\\
						\end{split}
					\end{equation}
					Now, from Remark~\ref{rem1}, $\emph{cent}_{C(S)}(F) = S_{\emph{rad}_{C(S)}(F)}(F) \neq \es$. Since $\emph{rad}_{C(S)}(F) \leq  R$, $S_{R}(F) \neq \es$. By \cite[Proposition~4.2]{CS}, $Y$ is ball proximinal in $C(S)$. Therefore, for each $f \in F$, $B[f, d(f,B_{Y})] \cap B_{Y} \neq \es$. It follows that for each $f \in F$, $B[f,R] \cap B_{C(S)}  \neq \es$. Since $C(S)$ is an $L_{1}$-predual space and $F$ is compact, by \cite[Theorem~4.5, pg.~38]{JL}, $S_{R}(F) \cap B_{C(S)} \neq \es.$ Let $g_{0} \in B_{C(S)} \cap S_{R}(F)$. Then for each $f \in F$ and $t \in S$, 
					\begin{equation}\label{eqn2.7}
						f(t)-R \leq g_{0}(t) \leq f(t)+R.
					\end{equation}	
					It follows that for $t \in S$, 
					\begin{equation}\label{eqn2.8}
						\sup_{f \in F} f(t)-R \leq \inf_{f \in F} f(t)+R.
					\end{equation}
					It also follows from ($\ref{eqn2.7}$) that for each $f \in F$ and $t \in S$, 
					\begin{equation}\label{eqn2.9}
						-1 - R \leq f(t) \leq R+ 1.
					\end{equation}	
					Now, choose $g \in B_{C(S)}$ such that $g(k_{i})=\eta_{i}$, for $i=1,\ldots,m$ and $g(t_{j})=\eta^{\prime}_{j}$, for $j=1,\ldots,r$. Let $h_{0}: S \rightarrow \mb{R}$ be defined as $h_{0} = \min\{g,\inf_{f \in F} f+R\}$. The compactness of $F$ ensures $h_{0} \in C(S)$. Further, define $h: S \rightarrow \mb{R}$ as $h = \max\{h_{0},\sup_{f \in F} f-R\}$. Then from the inequalities in ($\ref{eqn2.6}$), ($\ref{eqn2.8}$) and ($\ref{eqn2.9}$), it follows that $h \in B_{C(S)}$; $h(k_{i})=\eta_{i}$, for $i=1,\ldots,m$; $h(t_{j})= \eta^{\prime}_{j}$, for $j=1,\ldots,r$ and for each $t \in S$, $\sup_{f \in F} f(t)-R \leq h(t) \leq \inf_{f \in F} f(t)+R.$ Therefore, $h \in \emph{cent}_{B_{Y}}(F)$.
					
					Now, we prove that $(C(S), B_{Y}, \{F\})$ satisfies property-$(P_{1})$. Let $\e>0$. Let $X= \mb{R}^{m+r}$, equipped with the supremum norm and \[\tilde{F} = \{x_{f}=(f(k_{1}),\ldots,f(k_{m}),f(t_{1}),\ldots,f(t_{r})) \in X: f \in F\} \in \mc{K}(X).\]
					
					{\sc Subcase 1:} $R=\al$.
					
					Due to the compactness of the set~$A$, $(X,A,\mc{CB}(X))$ has property-$(P_{1})$. Hence, there exists $0<\de<\e$ such that $\emph{cent}_{A}(\tilde{F},\de) \ci \emph{cent}_{A}(\tilde{F}) + \e B_{X}$.
					
					Let $g \in \emph{cent}_{B_{Y}}(F,\de)$. Then $x_{g} = (g(k_{1}),\ldots,g(k_{m}),g(t_{1}),\ldots,g(t_{r})) \in \emph{cent}_{A}(\tilde{F},\de)$. Therefore, there exists $z=(z_{1},\ldots,z_{m},z^{\prime}_{1},\ldots,z^{\prime}_{r}) \in \emph{cent}_{A}(\tilde{F})$ such that $\|x_{g}-z\| \leq \e$. Now, choose $g^{\prime} \in B_{C(S)}$ such that $g^{\prime}(k_{i})=z_{i}$, for $i=1,\ldots,m$ and $g^{\prime}(t_{j})= z^{\prime}_{j}$, for $j=1,\ldots,r$. Let $f_{1} =\max \{\sup_{f \in F} f-R, g-\e, -1\}$ and $f_{2} = \min \{\inf_{f \in F} f+R, g+\e, 1\}$. Then $f_{1} \leq g^{\prime} \leq f_{2}$ on $\{k_{1},\ldots,k_{m},t_{1},\ldots,t_{r}\}$. Let $h_{1} = \max\{f_{1},g^{\prime}\}$ and $h_{2} = \min\{h_{1},f_{2}\}$. Since $r(g,F) \leq R+\de < R+\e$. It follows that $\sup_{f \in F} f-R \leq g+\e$ and $g-\e \leq \inf_{f \in F} f+R$. Also, from the inequalities in $(\ref{eqn2.9})$, it follows that $\sup_{f \in F} f-R \leq 1$ and $-1 \leq  \inf_{f \in F} f+R$. Further, since $g \in B_{Y}$, $-1 \leq g \leq 1$ and hence, $g-\e \leq 1$. Therefore, $f_{1} \leq f_{2}$ and $f_{1} \leq h_{1}$. We can then conclude that $h_{2} = g^\prime$ on $\{k_{1},\ldots,k_{m},t_{1},\ldots,t_{r}\}$ and $f_{1} \leq h_{2} \leq f_{2}$ on $S$. Therefore, $h_{2} \in B_{Y}$, $\sup_{f \in F} f-R \leq h_{2} \leq \inf_{f \in F} f+R$ and $g-\e \leq h_{2} \leq g+\e$. This implies $h_{2} \in \emph{cent}_{B_{Y}}(F)$ and $\|g-h_{2}\| \leq \e$. Hence, $(C(S), B_{Y},\{F\})$ satisfies property-$(P_{1})$.
					
					{\sc Subcase 2:} $R>\al$.
					
					Let $\be = R-\al$. By Lemma~\ref{lem1}, there exists $0<\de<\e$ such that $\emph{cent}_{A}(\tilde{F},\be+\de) \ci \emph{cent}_{A}(\tilde{F},\be)  +\e B_{X}$. 
					
					Let $g \in \emph{cent}_{B_{Y}}(F,\de)$. Then $x_{g} = (g(k_{1}),\ldots,g(k_{m}),g(t_{1}),\ldots,g(t_{r})) \in \emph{cent}_{A}(\tilde{F},\be +\de)$. Therefore, there exists $z =(z_{1},\ldots,z_{m},z^{\prime}_{1},\ldots,z^{\prime}_{r}) \in \emph{cent}_{A}(\tilde{F},\be)$ such that $\|x_{g}-z\| \leq \e$. Therefore, $r(z,\tilde{F}) \leq \al+\be = R$. Now, choose $g^{\prime} \in B_{C(S)}$ such that $g^{\prime}(k_{i})=z_{i}$ and $g^{\prime}(t_{j})= z^{\prime}_{j}$, for $i=1,\ldots,m$ and $j=1,\ldots,r$. Then by following the same steps as in the last paragraph of {\sc Subcase 1}, we can prove that $(C(S),B_{Y},\{F\})$ satisfies property-$(P_{1})$.
					
					{\sc Case 2:} $S(\mu_{1}) \cap S(\mu_{2}) \neq \es$.
					
					Without loss of generality, for simplicity, we assume that $S(\mu_{1}) \cap S(\mu_{2}) = \{k_{1},\ldots,k_{s}\}$, where $k_{i}=t_{i}$, for $1 \leq i \leq s \leq \min\{m,r\}$. Let us define 
					\begin{equation}\label{eqn2.10}
						\begin{split}
							B = \{(\ga_{1},\ldots,\ga_{m},\ga^{\prime}_{1},\ldots,\ga^{\prime}_{r}) \in [-1,1]^{m+r}: \ga_{i} = \ga^{\prime}_{i}\mbox{ for }1 \leq &i \leq s;\sum_{i=1}^{m} \al_{i} \ga_{i} = 0\\ &\mbox{ and }\sum_{j=1}^{r} \be_{j} \ga^{\prime}_{j}=0 \}
						\end{split}
					\end{equation}
					and 
					\begin{equation}\label{eqn2.11}
						\al^{\prime} = \inf \left\{\sup_{f \in F} \max_{\substack{1 \leq i \leq m\\ 1 \leq j \leq r}} \{|\ga_{i}-f(k_{i})|, |\ga^{\prime}_{j} - f(t_{j})|\}: (\ga_{1},\ldots,\ga_{m},\ga^{\prime}_{1},\ldots,\ga^{\prime}_{r}) \in B\right\}.
					\end{equation}
					Applying the same argument as in {\sc Case 1}, we can show that the infimum in ($\ref{eqn2.11}$) is attained say at $(\eta_{1},\ldots,\eta_{m},\eta^{\prime}_{1},\ldots,\eta^{\prime}_{r}) \in B$. We further proceed the same way as in {\sc Case 1} to first prove that $\emph{cent}_{B_{Y}}(F) \neq \es$ and then that $(C(S),B_{Y},\{F\})$ satisfies property-$(P_{1})$.
				\end{proof}
				
				We now prove our main result.				
				\begin{thm}\label{thm2.1}
					Let $S$ be a compact Hausdorff space and $Y$ be a finite co-dimensional subspace of $C(S)$. Then the following statements are equivalent :
					\blr
					\item $Y$ is strongly proximinal in $C(S)$.
					\item $Y$ is strongly ball proximinal in $C(S)$.
					\item $(C(S),Y,\mc{K}(C(S)))$ has property-$(P_{1})$.
					\item $(C(S),B_{Y},\mc{K}(C(S)))$ has property-$(P_{1})$.
					\item $Y^{\perp} \ci \{\mu \in C(S)^{\ast}: \mu\mbox{ is an SSD-point of }C(S)^{\ast}\}$.
					\el
				\end{thm}
				\begin{proof}
					By \cite[Theorem~4.3]{CS}, $(i) \Leftrightarrow (ii) \Leftrightarrow (v)$. The implication $(v) \Rightarrow (iv)$ follows from \cite[Theorem~2.1]{SD} and Theorem~\ref{thm2}. Also, clearly, $(iii) \Rightarrow (i)$ and from Proposition~\ref{prop0.2}, $(iv) \Rightarrow (iii)$.
				\end{proof}
				
				For a Choquet simplex~$K$ and a finite co-dimensional subspace $Y$ of $A(K)$, the following result provides a sufficient condition for the triplet $(A(K), B_{Y},\mc{K}(A(K)))$ to satisfy property-$(P_{1})$. The convex hull of a non-empty subset $A$ of $K$ is denoted by $conv(A)$.
				
				\begin{thm}\label{thm3}
					Let $K$ be a Choquet simplex and $\{\mu_{1},\ldots,\mu_{n}\} \ci A(K)^{\ast}$ such that for each $i=1,\ldots,n$, $\|\mu_i\|=1$. If for each $i=1,\ldots,n$, $S(\mu_{i})$ is finite, $S(\mu_{i}) \ci ext(K)$ and $Y=\bigcap_{i=1}^{n} ker(\mu_{i})$, then $(A(K),B_{Y},\mc{K}(A(K)))$ has property-$(P_{1})$.
				\end{thm}
				\begin{proof}
					We employ techniques similar to those used in the proof of \cite[Theorem~5.4]{CS}. We prove the result only for $n=2$ because the same ideas work to prove the result for $n \neq 2$. Let $\mu_{1} = \sum_{i=1}^{m} \al_{i} \de_{k_{i}}$ and $\mu_{2} = \sum_{j=1}^{r} \be_{j} \de_{t_{j}}$ and $Y = ker(\mu_{1}) \cap ker(\mu_{2})$. Let $F \in \mc{K}(A(K))$.
					
					{\sc Case 1:} $S(\mu_{1}) \cap S(\mu_{2}) =\es$.
					
					Let $\al, A$ be as defined in the proof of {\sc Case 1} of Theorem~\ref{thm2} and following the same argument as in that proof, let $(\eta_{1},\ldots,\eta_{m},\eta^{\prime}_{1},\ldots,\eta^{\prime}_{r}) \in A$ be such that \[\al = \sup_{f \in F} \max_{\substack{1 \leq i \leq m\\ 1 \leq j \leq r}} \{|\eta_{i}-f(k_{i})|, |\eta^{\prime}_{j} - f(t_{j})|\}.\] Let $R = \emph{rad}_{B_{Y}}(F)$. Then from the definition of $\al$ it follows that $R \geq \al$ and hence for each $f \in F$, 
					\begin{equation}\label{eqn3.1}
						\begin{split}
							&-R+ \eta_{i} \leq f(k_{i}) \leq R+ \eta_{i}\mbox{ for }i=1,\ldots,m\mbox{ and }\\ &-R+ \eta^{\prime}_{j} \leq f(t_{j}) \leq R+ \eta^{\prime}_{j}\mbox{ for }j=1,\ldots,r.\\
						\end{split}
					\end{equation} 
					It follows from Remark~\ref{rem1} that $\emph{cent}_{A(K)}(F) = S_{\emph{rad}_{A(K)}(F)}(F) \neq \es$. Since $\emph{rad}_{A(K)}(F) \leq R$, $S_{R}(F) \neq \es$. By \cite[Theorem~5.4]{CS}, $Y$ is ball proximinal in $A(K)$. Therefore, for each $f \in F$, $B[f, d(f,B_{Y})] \cap B_{Y} \neq \es$. For each $f \in F$, since $d(f,B_{Y}) \leq R$, it follows that $B_{A(K)} \cap B[f,R] \neq \es$. Hence, by \cite[Theorem~4.5, pg.~38]{JL}, $B_{A(K)} \cap S_{R}(F) \neq \es.$ Let $g_{0} \in B_{A(K)} \cap S_{R}(F)$. Then for each $f \in F$ and $t \in K$, 
					\begin{equation}\label{eqn3.2}
						f(t)-R \leq g_{0}(t) \leq f(t)+R.
					\end{equation}	
					It follows that for $t \in K$, 
					\begin{equation}\label{eqn3.3}
						\sup_{f \in F} f(t)-R \leq \inf_{f \in F} f(t)+R.
					\end{equation}
					It also follows from ($\ref{eqn3.2}$) that for each $f \in F$ and $t \in K$, 
					\begin{equation}\label{eqn3.4}
						-1 - R \leq f(t) \leq R+ 1.
					\end{equation}	
					Let us choose $g \in B_{C(K)}$ such that $g(k_{i})=\eta_{i}$, for $i=1,\ldots,m$ and $g(t_{j}) = \eta^{\prime}_{j}$, for $j=1,\ldots,r$. Define $h_{0}: K \rightarrow \mb{R}$ as follows: for each $t \in K$,
					\[h_{0}(t) =  \begin{cases} 
						\inf_{f \in F} f(t) + R &\mbox{, if }g(t) \geq \inf_{f \in F} f(t) +R \\
						g(k) &\mbox{, if }\sup_{f \in F} f(t) -R  \leq g(t) \leq \inf_{f \in F} f(t) +R\\
						\sup_{f \in F} f(t) -R &\mbox{, if }g(t) \leq \sup_{f \in F} f(t) -R.\\ 
					\end{cases}
					\]
					The compactness of $F$ and the inequalities in ($\ref{eqn3.4}$) ensure $h_{0} \in B_{C(K)}$. By the definition of $h_{0}$, $\sup_{f \in F} f-R \leq h_{0} \leq \inf_{f \in F} f+ R$ on $K$. From the inequalities in ($\ref{eqn3.1}$), it follows that for $i=1,\ldots,m$, $h_{0}(k_{i}) =\eta_{i}$ and for $j=1,\ldots,r$, $h_{0}(t_{j})=\eta^{\prime}_{j}$. Hence, $\sum_{i=1}^{m} \al_{i} h_{0}(k_{i}) =0 = \sum_{j=1}^{r} \be_{j} h_{0}(t_{j})$.
					
					Now, by \cite[Theorem~II.3.12]{AL}, there exists $h \in B_{A(K)}$ such that for each $i=1,\ldots,m$ and $j=1,\ldots,r$, $h(k_{i}) =h_{0}(k_{i})$ and $h(t_{j}) = h_{0}(t_{j})$. Let $G = conv(\{k_{1},\ldots,k_{m},t_{1},\ldots,t_{r}\})$. Then $G$ is a closed face of $K$. Further, for each $f \in F$, $f-R \leq h \leq f+R$ on $G$ and hence, $\sup_{f \in F} f-R \leq h \leq \inf_{f \in F} f+R$ on $G$. Also, $-1 \leq h \leq 1$ on $G$. Therefore, from the inequalities in ($\ref{eqn3.4}$), it follows that \[\max\left\{-1, \sup_{f \in F} f-R\right\} \leq h \leq \min\left\{1, \inf_{f \in F} f+R\right\}\mbox{ on }G\] and \[\max\left\{-1, \sup_{f \in F} f-R\right\} \leq \min\left\{1, \inf_{f \in F} f+R\right\}\mbox{ on }K.\] Note that $\max\left\{-1, \sup_{f \in F} f-R\right\}$ and $-\min\left\{1, \inf_{f \in F} f+R\right\}$ are convex continuous functions on $K$. Therefore, by \cite[Corollary~7.7, p.~73]{AE}, there exists $\tilde{h} \in A(K)$ such that $\tilde{h} = h$ on $G$ and $\max\left\{-1, \sup_{f \in F} f-R\right\} \leq \tilde{h} \leq \min\left\{1, \inf_{f \in F} f+R\right\}$ on $K$. It follows that $\tilde{h} \in \emph{cent}_{B_{Y}}(F)$.
					
					Now, we prove that $(A(K),B_{Y},\{F\})$ satisfies property-$(P_{1})$. Let $\e>0$. Let $X =\mb{R}^{m+r}$, equipped with the supremum norm and \[\tilde{F} = \{x_{f}=(f(k_{1}),\ldots,f(k_{m}),f(t_{1}),\ldots,f(t_{r})) \in X: f \in F\} \in \mc{K}(X).\]
					
					{\sc Subcase 1:} $R=\al$.
					
					The set~$A \ci X$ is compact and hence, $(X,A,\mc{CB}(X))$ has property-$(P_{1})$. Therefore, there exists $0<\de<\e$ such that $\emph{cent}_{A}(\tilde{F},\de) \ci \emph{cent}_{A}(\tilde{F}) + \e B_{X}$.
					
					Let $g \in \emph{cent}_{B_{Y}}(F,\de)$. Then $x_{g} = (g(k_{1}),\ldots,g(k_{m}),g(t_{1}),\ldots,g(t_{r})) \in \emph{cent}_{A}(\tilde{F},\de)$. Therefore, there exists $z=(z_{1},\ldots,z_{m},z^{\prime}_{1},\ldots,z^{\prime}_{r}) \in \emph{cent}_{A}(\tilde{F})$ such that $\|x_{g}-z\| \leq \e$. Now, choose $g^{\prime} \in B_{C(K)}$ such that $g^{\prime}(k_{i})=z_{i}$ and $g^{\prime}(t_{j})= z^{\prime}_{j}$, for $i=1,\ldots,m$ and $j=1,\ldots,r$. Then by \cite[Theorem~II.3.12]{AL}, there exists $h^{\prime} \in B_{A(K)}$ such that $h^{\prime}(k_{i}) =g^{\prime}(k_{i})=z_{i}$, for $i=1,\ldots,m$ and $h^{\prime}(t_{j}) =g^{\prime}(t_{j})=z^{\prime}_{j}$, for $j=1,\ldots,r$. Therefore, $\sum_{i=1}^{m} \al_{i} h^{\prime}(k_{i}) = 0 = \sum_{j=1}^{r} \be_{j} h^{\prime}(t_{j})$.
					
					Let $G = conv(\{k_{1},\ldots,k_{m},t_{1},\ldots,t_{r}\})$. Then $G$ is a closed face of $K$. Clearly, $\sup_{f \in F} f-R \leq h^{\prime} \leq \inf_{f\in F} f+R$ on $G$, $g-\e \leq h^{\prime} \leq g+\e$ on $G$ and $-1 \leq h^{\prime} \leq 1$ on $G$. Since $r(g,F) \leq R+\de< R+\e$, it follows that $\sup_{f \in F} f-R \leq g+\e$ on $K$ and $g-\e \leq \inf_{f \in F} f+R$ on $K$. Since $g \in B_{Y}$, $-1 \leq g \leq 1$ and hence $g-\e\leq 1$ on $K$. Therefore, \[\max\left\{\sup_{f \in F} f-R, g-\e, -1\right\} \leq h^{\prime} \leq \min\left\{\inf_{f \in F} f+R, g+\e,1\right\}\mbox{ on }G\] and \[\max\left\{\sup_{f \in F} f-R, g-\e, -1\right\} \leq \min\left\{\inf_{f \in F} f+R, g+\e,1\right\}\mbox{ on }K.\] Also, note that $\max\{\sup_{f \in F} f-R, g-\e, -1\}$ and $-\min\{\inf_{f \in F} f+R, g+\e,1\}$ are convex continuous functions on $K$. Therefore, by \cite[Corollary~7.7, p.~73]{AE}, there exists $h \in A(K)$ such that $h = h^{\prime}$ on $G$ and \[\max\{\sup_{f \in F} f-R, g-\e, -1\} \leq h \leq \min\{\inf_{f \in F} f+R, g+\e,1\}\mbox{ on }K.\] It follows that $h \in \emph{cent}_{B_{Y}}(F)$ such that $\|g-h\| \leq \e$. Hence, $(A(K),B_{Y},\{F\})$ satisfies property-$(P_{1})$.
					
					{\sc Subcase 2:} $R>\al$.
					
					Let $\be = R-\al$. By Lemma~\ref{lem1}, there exists $0<\de<\e$ such that $\emph{cent}_{A}(\tilde{F},\be+\de) \ci \emph{cent}_{A}(\tilde{F},\be)  +\e B_{X}$. 
					
					Let $g \in \emph{cent}_{B_{Y}}(F,\de)$. Then $x_{g} = (g(k_{1}),\ldots,g(k_{m}),g(t_{1}),\ldots,g(t_{r})) \in \emph{cent}_{A}(\tilde{F},\be +\de)$. Therefore, there exists $z =(z_{1},\ldots,z_{m},z^{\prime}_{1},\ldots,z^{\prime}_{r}) \in \emph{cent}_{A}(\tilde{F},\be)$ such that $\|x_{g}-z\| \leq \e$. Therefore, $r(z,\tilde{F}) \leq \al+\be = R$. Now, choose $g^{\prime} \in B_{C(K)}$ such that $g^{\prime}(k_{i})=z_{i}$ and $g^{\prime}(t_{j})= z^{\prime}_{j}$, for $i=1,\ldots,m$ and $j=1,\ldots,r$. Therefore, by \cite[Theorem~II.3.12]{AL}, there exists $h^{\prime} \in B_{A(K)}$ such that $h^{\prime}(k_{i}) =g^{\prime}(k_{i})=z_{i}$, for $i=1,\ldots,m$ and $h^{\prime}(t_{j}) =g^{\prime}(t_{j})=z^{\prime}_{j}$, for $j=1,\ldots,r$. Then by following the same steps as in the last paragraph of {\sc Subcase 1}, we can prove that $(A(K),B_{Y},\{F\})$ satisfies property-$(P_{1})$.
					
					{\sc Case 2:} $S(\mu_{1}) \cap S(\mu_{2}) \neq \es$.
					
					Without loss of generality, for simplicity, we assume that $S(\mu_{1}) \cap S(\mu_{2}) = \{k_{1},\ldots,k_{s}\}$, where $k_{i}=t_{i}$, for $1 \leq i \leq s \leq \min\{m,r\}$. Let $B$ and $\al^{\prime}$ be defined as in the proof of {\sc Case 2} of Theorem~\ref{thm2}. We further proceed the same way as in {\sc Case 1} to prove that $(A(K),B_{Y},\{F\})$ satisfies property-$(P_{1})$.
				\end{proof}
				
				The following result is an easy consequence of \cite[Theorem~5.3]{CS}, Theorem~\ref{thm3}, Proposition~\ref{prop0.2} and \cite[Theorem~2.6]{C}.
				\begin{thm}\label{thm3.1}
					Let $K$ be a Choquet simplex; $\{\mu_{1},\ldots,\mu_{n}\} \ci A(K)^{\ast}$ be such that for each $i=1,\ldots,n$, $S(\mu_i) \ci ext(K)$ and $Y = \bigcap_{i=1}^{n} ker(\mu_i)$. Then the following are equivalent:
					\blr
					\item $Y$ is strongly proximinal in $A(K)$.
					\item $Y$ is strongly ball proximinal in $A(K)$.
					\item $(A(K),Y,\mc{K}(A(K)))$ has property-$(P_{1})$.
					\item $(A(K),B_{Y},\mc{K}(A(K)))$ has property-$(P_{1})$.
					\item $Y^{\perp} \ci \{\mu \in A(K)^{\ast}: \mu\mbox{ is an SSD-point of }A(K)^{\ast}\}$.
					\el
				\end{thm} 
				
				
				\section{Characterization of strongly proximinal finite co-dimensional subspaces of $L_{1}$-predual spaces in terms of property-$(P_{1})$}\label{sec4}
				In this section, our main aim is to generalize the characterization in Theorem~\ref{thm2.1} for the strongly proximinal finite co-dimensional subspaces of an $L_{1}$-predual space. To this end, we need few technical lemmas.
				
				For a Banach space $X$, the Hausdorff metric, denoted by $d_{H}$, on $\mc{CB}(X)$ is defined as follows: for each $B_1, B_2 \in \mc{CB}(X)$, \[d_{H}(B_1,B_2) = \inf \{a>0: B_1 \ci B_2 + a B(0,1)\mbox{ and }B_2 \ci B_1 + aB(0,1)\}.\] The following lemma is proved in \cite[Theorem~2.5]{ST}. We include the proof here for the sake of completeness. 
				\begin{lem}\label{lem3}
					Let $V$ be a non-empty closed convex subset of a Banach space~$X$ and $F_{1},F_{2} \in \mc{CB}(X)$. Then for each $v \in V$, $|r(v,F_{1})-r(v,F_{2})| \leq d_{H}(F_{1},F_{2})$ and $|\textrm{rad}_{V}(F_{1})-\textrm{rad}_{V}(F_{2})| \leq d_{H}(F_{1},F_{2}).$
				\end{lem}
				\begin{proof}
					Let $v \in V$. Now, let $y \in F_{1}$ and $\e>0$. Choose $z \in F_{2}$ such that $\|y-z\|<d_{H}(F_{1},F_{2})+\e$. Then \[\|v-y\| \leq \|v-z\|+\|z-y\|<r(v,F_{2})+d_{H}(F_{1},F_{2})+\e.\] It follows that 
					\begin{equation}\label{eqnlem3.0}
						r(v,F_{1}) \leq r(v,F_{2})+d_{H}(F_{1},F_{2}).
					\end{equation} 
					Further, after swapping $F_{1}$ with $F_{2}$ in the above argument, we obtain the following inequality. 
					\begin{equation}\label{eqnlem3.1}
						r(v,F_{2}) \leq r(v,F_{1})+d_{H}(F_{1},F_{2}).
					\end{equation} 
					The first conclusion of the result follows from the inequalities in $(\ref{eqnlem3.0})$ and $(\ref{eqnlem3.1})$.
					
					The inequalities in $(\ref{eqnlem3.0})$ and $(\ref{eqnlem3.1})$ hold true for every $v \in V$ and hence, the final conclusion of the result follows. 
				\end{proof}
				
				\begin{lem}\label{lem2}
					Let $Y$ be a subspace of a Banach space. Then for each $F \in \mc{K}(X)$, $\textrm{rad}_{B_{Y^{\perp \perp}}}(F) = \textrm{rad}_{B_{Y}}(F)$.  	
				\end{lem}
				\begin{proof}
					First we prove the result for each set in $\mc{F}(X)$. Let $F = \{x_1,\ldots,x_n\} \in \mc{F}(X)$. Clearly, $\emph{rad}_{B_{Y^{\perp \perp}}}(F) \leq \emph{rad}_{B_Y}(F)$. Suppose $\emph{rad}_{B_{Y^{\perp \perp}}}(F) < \emph{rad}_{B_Y}(F)$. Let us choose $\e>0$ and $\Phi \in B_{Y^{\perp \perp}}$ such that $r(\Phi, F)<\emph{rad}_{B_Y}(F)-\e$. Now, choose $0< \e^\prime < \frac{\e}{1+r(\Phi,F)}$ and define $E=span \{x_1,\ldots,x_n, \Phi\} \ci X^{\ast \ast}$. Then by the extended version of principle of local reflexivity in \cite[Theorem~3.2]{B}, there exists a bounded linear map~$T : E \rightarrow X$ such that $T(x_i)=x_i$, for each $i=1,\ldots,n$; $T(\Phi) \in Y$ and $\|T\| \leq 1+\e^{\prime}$. Let $y = \frac{T(\Phi)}{1+\e^{\prime}} \in B_Y$. Then for each $i=1,\ldots,n$, 
					\begin{equation}
						\begin{split}
							\|x_i-y\| &\leq \|T(x_i) - T(\Phi)\| + \left\|T(\Phi) - \frac{T(\Phi)}{1+\e^{\prime}}\right\| \\ &\leq (1+\e^\prime)\|x_i - \Phi\| + \e^{\prime}\\ &\leq r(\Phi,F)  +\e^\prime (1+r(\Phi,F))\\& < r(\Phi,F) + \e.\\
						\end{split}
					\end{equation} 
					It follows that $r(y,F) \leq r(\Phi,F) + \e$. Now, from the inequalities $\emph{rad}_{B_Y}(F) \leq r(y,F)$ and $r(\Phi,F) < \emph{rad}_{B_{Y}}(F)-\e$, it follows $\emph{rad}_{B_{Y}}(F) < \emph{rad}_{B_{Y}}(F)$, which is a contradiction. Therefore, $\emph{rad}_{B_{Y^{\perp \perp}}}(F) = \emph{rad}_{B_{Y}}(F)$.  	
					
					Now, for a set $F \in \mc{K}(X)$, it follows from Lemma~\ref{lem3}; the fact that for each $\e>0$, there exists a finite $\e$-net $F_{\e}$ such that $d_{H}(F_{\e},F)<\e$ and the first part of the proof that $\emph{rad}_{B_{Y^{\perp \perp}}}(F) = \emph{rad}_{B_{Y}}(F)$.  	
				\end{proof}	
				
				\begin{lem}\label{lem4}
					Let $Y$ be a subspace of a Banach space~$X$. If $(X^{\ast \ast},B_{Y^{\perp \perp}},\mc{K}(X))$ has property-$(P_{1})$, then for each $F \in \mc{K}(X)$ and $y \in Y$, $d(y,\textrm{cent}_{B_{Y^{\perp \perp}}}(F)) = d(y,\textrm{cent}_{B_{Y}}(F))$.
				\end{lem}
				\begin{proof}
					We follow the proof technique of \cite[Lemma~2.2]{C}. Let $F \in \mc{K}(X)$ and $y \in Y$. Define $r = d(y,\emph{cent}_{B_{Y^{\perp \perp}}}(F))$ and $r^{\prime} = \emph{rad}_{B_{Y}}(F)$. By Lemma~\ref{lem2}, $r^{\prime} = \emph{rad}_{B_{Y^{\perp \perp}}}(F)$ and hence, for each $\de>0$, $\emph{cent}_{B_{Y}}(F,\de) \ci \emph{cent}_{B_{Y^{\perp \perp}}}(F,\de)$. Therefore, by our assumption, for each $\e>0$, there exists $\de_{\e}>0$ such that $d(v,\emph{cent}_{B_{Y^{\perp \perp}}}(F))<\e$, whenever $v \in \emph{cent}_{B_{Y}}(F,\de_{\e})$.
					
					Now, let $\e>0$ be fixed. 
					
					Let us choose $0<\be < \frac{\e}{3}$ and define $\de=\de_{\frac{\e}{2^{2}}}$. For each $m\in \mb{N}$, let $F_{m} \ci F$ be finite $\frac{\de}{2^{m+2}}$-net such that $F_{m} \ci F_{m+1}$ and define $r^{\prime}_{m}=\emph{rad}_{B_{Y}}(F_{m})$. By Lemma~\ref{lem2}, $r^{\prime}_{m}=\emph{rad}_{B_{Y^{\perp \perp}}}(F_{m})$. Therefore, by Lemma~\ref{lem1}, for each $m \in \mb{N}$ and $\e^{\prime}>0$, there exists $0<\ga^{m}_{\e^{\prime}}< \frac{\de}{2}$ such that $d(v,\emph{cent}_{B_{Y^{\perp \perp}}}(F_{m},\sum_{k=1}^{m} \frac{\de}{2^{k+1}}))< \e^{\prime}$, whenever $v \in \emph{cent}_{B_{Y}}(F_{m},\sum_{k=1}^{m} \frac{\de}{2^{k+1}} + \ga^{m}_{\e^{\prime}})$. 
					
					Now, since $\emph{cent}_{B_{Y^{\perp \perp}}}(F_{1},\frac{\de}{2^2})$ is $w^{\ast}$-compact, it is proximinal and hence there exists $\Phi_{0} \in \emph{cent}_{B_{Y^{\perp \perp}}}(F_{1},\frac{\de}{2^2})$ such that $d(y,\emph{cent}_{B_{Y^{\perp \perp}}}(F_{1},\frac{\de}{2^2}))=\|y-\Phi_{0}\|$. Define $r_{0} = d(y,\emph{cent}_{B_{Y^{\perp \perp}}}(F_{1},\frac{\de}{2^2}))$.
					It is easy to see that $\emph{cent}_{B_{Y^{\perp \perp}}}(F) \ci \emph{cent}_{B_{Y^{\perp \perp}}}(F_{1},\frac{\de}{2^2})$. Indeed, it follows from Lemma~\ref{lem3} that $r^\prime \leq r^{\prime}_{1}+ \frac{\de}{2^{3}}$ and hence, for $\Phi^\prime \in \emph{cent}_{B_{Y^{\perp \perp}}}(F)$, $r(\Phi^\prime,F_{1})\leq r(\Phi^\prime,F)=r^\prime \leq r^{\prime}_{1}+ \frac{\de}{2^{3}}<r^{\prime}_{1} + \frac{\de}{2^2}.$ Therefore, it follows that $r_{0} \leq r$.
					
					Choose $0<\e_{1}< \min\left\{\frac{3\be}{2^2 (r_{0}+1)}, \frac{\ga^{1}_{\frac{\be}{2^2}}}{1+r^{\prime}_{1}+\frac{\de}{2^{2}}}\right\}$. Let $E_{1}=span \{F_{1}\cup \{y,\Phi_{0}\}\} \ci X^{\ast \ast}$. Then by the extended version of principle of local reflexivity in \cite[Theorem~3.2]{B}, there exists a bounded linear map~$T_{1} : E_{1} \rightarrow X$ such that $T_{1}(x)=x$, for each $x \in F_{1}$; $T_{1}(y) = y$; $T_{1}(\Phi_{0}) \in Y$ and $\|T_{1}\| \leq 1+\e_{1}$. Now, let $y_{1} = \frac{T_{1}(\Phi_{0})}{1+\e_{1}} \in B_{Y}$. Then 
					\begin{equation}
						\begin{split}
							\|y-y_{1}\| &\leq \|T_{1}(y) - T_{1}(\Phi_{0})\| + \left\|T_{1}(\Phi_{0}) -\frac{T_{1}(\Phi_{0})}{1+\e_{1}}\right\|\\ &\leq (1+\e_{1})r_{0} + \e_{1} \\&\leq  r + \e_{1} (1+r_{0})\\ &< r+ \frac{3\be}{2^2}.\\ 
						\end{split}
					\end{equation}
					Also, for each $x \in F_{1}$,
					\begin{equation}
						\begin{split}
							\|x-y_{1}\| &\leq \|T_{1}(x) - T_{1}(\Phi_{0})\| + \left\|T_{1}(\Phi_{0}) -\frac{T_{1}(\Phi_{0})}{1+\e_{1}}\right\| \\ &\leq (1+\e_{1})r(\Phi_{0},F_{1})  + \e_{1}\\ &\leq r^{\prime}_{1} +\frac{\de}{2^{2}} + \e_{1}\left(1+r^{\prime}_{1}+\frac{\de}{2^{2}}\right) \\&< r^{\prime}_{1} + \frac{\de}{2^{2}} + \ga^{1}_{\frac{\be}{2^2}}.\\
						\end{split}
					\end{equation}
					It follows that $r(y_{1},F_{1}) \leq r^{\prime}_{1} + \frac{\de}{2^2} + \ga^{1}_{\frac{\be}{2^2}}.$ Thus, $y_{1} \in \emph{cent}_{B_{Y}}(F_{1},\frac{\de}{2^2} + \ga^{1}_{\frac{\be}{2^2}})$. This implies $d(y_{1},\emph{cent}_{B_{Y^{\perp \perp}}}(F_{1},\frac{\de}{2^2}))< \frac{\be}{2^2}$. Now, let $\Phi_{1} \in \emph{cent}_{B_{Y^{\perp \perp}}}(F_{1},\frac{\de}{2^2})$ such that $\|y_{1} -\Phi_{1}\|< \frac{\be}{2^2}$.
					
					Let us make the following observation. Let $x \in F_{2}$. Then there exists $x_{1} \in F_{1}$ such that $\|x-x_{1}\|< \frac{\de}{2^3}$ and hence,
					\begin{equation}
						\begin{split}
							\|x-\Phi_{1}\| &\leq \|x-x_{1}\|+\|x_{1}-\Phi_{1}\| \\ &< \frac{\de}{2^3} + r(\Phi_{1},F_{1})\\ &\leq  \frac{\de}{2^3} + r^{\prime}_{1} + \frac{\de}{2^2}\\ &\leq r^{\prime}_{2} + \frac{\de}{2^2} +\frac{\de}{2^3}.\\
						\end{split}
					\end{equation} 
					It follows that $r(\Phi_{1},F_{2}) \leq r^{\prime}_{2}  + \frac{\de}{2^2} + \frac{\de}{2^3}$. 
					
					Choose $0<\e_{2}< \min \left\{\frac{\be}{2^{3}(1+\frac{\be}{2^2})}, \frac{\ga^{2}_{\frac{\be}{2^3}}}{1+r^{\prime}_{2}+ \frac{\de}{2^2} + \frac{\de}{2^3}} \right\}$. Let $E_{2} = span \{F_{2} \cup \{\Phi_{1},y_{1}\}\} \ci X^{\ast \ast}$. Then, again by principle of local reflexivity, there exists a bounded linear map~$T_{2}: E_{2} \rightarrow X$ such that $T_{2}(x)=x$, for each $x \in F_{2}$; $T_{2}(y_{1}) =y_{1}$; $T_{2}(\Phi_{1}) \in Y$ and $\|T_{2}\| \leq 1+ \e_{2}$. Now, let $y_{2} = \frac{T_{2}(\Phi_{1})}{1+\e_{2}} \in B_{Y}$. Then 
					\begin{equation}
						\begin{split}
							\|y_{1}-y_{2}\| &\leq \|T_{2}(y_{1}) - T_{2}(\Phi_{1})\| + \left\|T_{2}(\Phi_{1}) -\frac{T_{2}(\Phi_{1})}{1+\e_{2}}\right\|\\ &< (1+\e_{2})\frac{\be}{2^2} + \e_{2} \\&= \frac{\be}{2^2} + \e_{2} \left(1+\frac{\be}{2^2}\right)\\ &< \frac{\be}{2^2}+ \frac{\be}{2^3}= \frac{3 \be}{2^3}.\\  
						\end{split}
					\end{equation}
					Also, for each $x \in F_{2}$,
					\begin{equation}
						\begin{split}
							\|x-y_{2}\| &\leq \|T_{2}(x) - T_{2}(\Phi_{1})\| + \left\|T_{2}(\Phi_{1}) -\frac{T_{2}(\Phi_{1})}{1+\e_{2}}\right\| \\ &\leq (1+\e_{2})r(\Phi_{1},F_{2})  + \e_{2}\\&\leq r^{\prime}_{2}  + \frac{\de}{2^2} + \frac{\de}{2^3} + \e_{2}\left(1+r^{\prime}_{2}  + \frac{\de}{2^2} + \frac{\de}{2^3}\right) \\&< r^{\prime}_{2}  + \frac{\de}{2^2} + \frac{\de}{2^3} + \ga^{2}_{\frac{\be}{2^3}}.\\
						\end{split}
					\end{equation}
					It follows that $r(y_{2},F_{2}) \leq r^{\prime}_{2}  + \frac{\de}{2^2} + \frac{\de}{2^3}+ \ga^{2}_{\frac{\be}{2^3}}.$ Thus, $y_{2} \in \emph{cent}_{B_{Y}}(F_{2},\frac{\de}{2^2} + \frac{\de}{2^3}+ \ga^{2}_{\frac{\be}{2^3}})$. This implies $d(y_{2},\emph{cent}_{B_{Y^{\perp \perp}}}(F_{2},\frac{\de}{2^2} + \frac{\de}{2^3}))< \frac{\be}{2^3}$. Now, let $\Phi_{2} \in \emph{cent}_{B_{Y^{\perp \perp}}}(F_{2},\frac{\de}{2^2} + \frac{\de}{2^3})$ such that $\|y_{2} -\Phi_{2}\|< \frac{\be}{2^3}$. 
					Similar to the earlier observation, we can conclude that $r(\Phi_{2},F_{3}) \leq r^{\prime}_{3}  + \frac{\de}{2^2} + \frac{\de}{2^3}+\frac{\de}{2^4}.$
					
					Proceeding inductively, we get a sequence~$\{y_{n}\} \ci B_{Y}$ such that $\|y_{n}-y_{n+1}\| < \frac{3 \be}{2^{n+2}}$ and $r(y_{n},F_{n}) \leq r^{\prime}_{n} + \sum_{k=1}^{n} \frac{\de}{2^{k+1}} + \ga^{n}_{\frac{\be}{2^{n+1}}}<r^{\prime} + \sum_{k=1}^{n} \frac{\de}{2^{k+1}} +\frac{\de}{2}.$ Clearly, $\{y_{n}\}$ is Cauchy in $B_{Y}$ and hence, let $z_{1} \in B_{Y}$ such that $z_{1} = \lim_{n \rightarrow \iy} y_{n}$. Then $\|y-z_{1}\| \leq r + \sum_{n=1}^{\iy} \frac{3 \be}{2^{n+1}} = r + \frac{3 \be}{2}< r + \frac{\e}{2}$. 
					
					Now, let $\e^{\prime}>0$ and $x \in F$. Then there exists $n_{0} \in \mb{N}$
					such that $\frac{\de}{2^{n_{0}}} <\frac{\e^{\prime}}{3}$, $\|y_{n_{0}}-z_{1}\|<\frac{\e^{\prime}}{3}$ and $\sum_{k=1}^{n_{0}} \frac{\de}{2^{k+1}} < \frac{\de}{2} + \frac{\e^{\prime}}{3}$ and $x_{n_{0}} \in F_{n_{0}}$ such that $\|x-x_{n_{0}}\|< \frac{\de}{2^{n_{0}+2}}$. Therefore, 
					\begin{equation}
						\begin{split}
							\|x-z_{1}\| &\leq \|x-x_{n_{0}}\| + \|x_{n_{0}} - y_{n_{0}}\| + \|y_{n_{0}}-z_{1}\| \\ &< \frac{\de}{2^{n_{0}+2}} + r(y_{n_{0}}, F_{n_{0}}) + \frac{\e^{\prime}}{3} \\& < \frac{\e^{\prime}}{3}  + r^{\prime} + \sum_{k=1}^{n_{0}} \frac{\de}{2^{k+1}}  + \frac{\de}{2} + \frac{\e^{\prime}}{3}\\ &< r^{\prime} + \de+\e^\prime.\\  
						\end{split}
					\end{equation}
					It follows that $r(z_{1},F) \leq r^{\prime} + \de + \e^{\prime}.$ Since $\e^{\prime}$ is arbitrary, $r(z_{1},F) \leq r^{\prime} + \de=r^{\prime} + \de_{\frac{\e}{2^2}}$.   
					
					Thus, $z_{1} \in \emph{cent}_{B_{Y}}(F,\de_{\frac{\e}{2^2}})$ and hence, $d(z_{1},\emph{cent}_{B_{Y^{\perp \perp}}}(F))<\frac{\e}{2^2}$. Now, for each $m \in \mb{N}$, choose a finite $\frac{\de_{\e/2^3}}{2^{m+2}}$-net $G_{m} \ci F$ such that $G_{m} \ci G_{m+1}$. Therefore, there exists $\psi \in \emph{cent}_{B_{Y^{\perp \perp}}}(G_1,\frac{\de_{\e/2^3}}{2^2})$ such that $\|z_{1}-\psi\|<\frac{\e}{2^2}$. Then by applying similar arguments as above, there exists an element~$z_{2} \in B_{Y}$ such that $\|z_{1}-z_{2}\| < \frac{\e}{2^2}$ and $r(z_{2},F) \leq r^{\prime} + \de_{\frac{\e}{2^3}}$. 
					
					Again, proceeding inductively, we get a sequence~$\{z_{n}\} \ci B_{Y}$ such that $\|z_{n}-z_{n+1}\| < \frac{\e}{2^{n+1}}$ and $r(z_{n},F) \leq r^\prime + \de_{\frac{\e}{2^{n+1}}}$. Without loss of generality, we assume $\de_{\frac{\e}{2^{n+1}}} \rightarrow 0$. Clearly, $\{z_{n}\}$ is Cauchy in $B_{Y}$ and hence, let $z_{0} \in B_{Y}$ such that $z_{0} = \lim_{n \rightarrow \iy} z_{n}$. Let $x \in F$. Then $\|x-z_{0}\| = \lim_{n \rightarrow \iy} \|x-z_{n}\| \leq \lim_{n \rightarrow \iy}  r(z_{n},F) = r^{\prime}$. It follows that $r(z_{0}, F) \leq r^\prime$ and hence, $z_{0} \in \emph{cent}_{B_{Y}}(F)$. Also, $\|y-z_{0}\| \leq r  + \sum_{n =1}^{\iy} \frac{\e}{2^{n}} = r+\e$. Therefore, $d(y,\emph{cent}_{B_{Y}}(F)) \leq \|y-z_{0}\| \leq d(y,\emph{cent}_{B_{Y^{\perp \perp}}}(F)) + \e$. Since $\e$ is arbitrary, $d(y,\emph{cent}_{B_{Y}}(F)) \leq d(y,\emph{cent}_{B_{Y^{\perp \perp}}}(F))$. This proves the result.
				\end{proof}	
				
				The following result connects property-$(P_{1})$ of the closed unit ball of a subspace of a Banach space with its bidual. It is proved using an argument similar to that in the proof of \cite[Proposition~2.3]{C}.  
				\begin{prop}\label{prop2}
					Let $Y$ be a subspace of a Banach space~$X$. If $(X^{\ast \ast},B_{Y^{\perp \perp}},\mc{K}(X))$ has property-$(P_{1})$, then $(X,B_{Y},\mc{K}(X))$ has property-$(P_{1})$.
				\end{prop}
				\begin{proof}
					Let $F \in \mc{K}(X)$. It follows from the proof of Lemma~\ref{lem4} that $\emph{cent}_{B_{Y}}(F) \neq \es$.  Now, let $\{y_{n}\}$ be a sequence in $B_{Y}$ such that $r(y_{n},F) \rightarrow \emph{rad}_{B_{Y}}(F)$. By Lemma~\ref{lem2}, $\emph{rad}_{B_{Y}}(F) = \emph{rad}_{B_{Y^{\perp \perp}}}(F)$. Therefore, $d(y_{n},\emph{cent}_{B_{Y^{\perp \perp}}}(F)) \rightarrow 0$. Hence, by Lemma~\ref{lem4}, $d(y_{n},\emph{cent}_{B_{Y}}(F)) \rightarrow 0$. Therefore, $(X,B_{Y},\{F\})$ satisfies property $(P_{1})$.
				\end{proof}	
				
				The next result characterizes property-$(P_1)$ of the closed unit ball of a finite co-dimensional subspace of an $L_1$-predual space in terms of property-$(P_1)$ of the closed unit ball of its bidual.
				\begin{prop}\label{prop3}
					Let $Y$ be a finite co-dimensional subspace of an $L_1$-predual space~$X$. Then $(X,B_{Y},\mc{K}(X))$ has property-$(P_{1})$ if and only if $(X^{\ast \ast},B_{Y^{\perp \perp}},\mc{K}(X^{\ast \ast}))$ has property-$(P_{1})$.
				\end{prop}
				\begin{proof}
					Assume $(X,B_{Y},\mc{K}(X))$ has property-$(P_{1})$. Then, in particular, $Y$ is strongly ball proximinal in $X$. Now, using an argument similar to that in the proof of \cite[Proposition~2.4]{C} and Theorem~\ref{thm2.1}, it follows that $(X^{\ast \ast},B_{Y^{\perp \perp}},\mc{K}(X^{\ast \ast}))$ has property-$(P_{1})$.
					
					The converse of the result follows from Proposition~\ref{prop2}.
				\end{proof}
				
				For a Banach space $X$, the result in \cite[Corollary~2.5]{C} shows strong ball proximinality through the weak$^{\ast}$-dense subset~$X$ in $X^{\ast \ast}$. In the following result, we demonstrate the same for property-$(P_{1})$ by following a similar argument.
				\begin{cor}\label{cor1}
					Let $X$ be an $L_1$-predual space and $Z$ be a finite co-dimensional weak$^{\ast}$-closed subspace of $X^{\ast \ast}$. If $(X^{\ast \ast}, B_{Z},\mc{K}(X))$ has property-$(P_{1})$, then so does $(X^{\ast \ast}, B_{Z},\mc{K}(X^{\ast \ast}))$.
				\end{cor}
				\begin{proof}
					Since $Z$ is a finite co-dimensional weak$^{\ast}$-closed subspace of $X^{\ast \ast}$, there exists a basis $\{x^{\ast}_{1},\ldots,x^{\ast}_{n}\} \ci X^\ast$ for $Z^\perp$. Now, let $Y = \bigcap_{i=1}^{n} ker(x^{\ast}_{i})$. Then $Y^{\perp \perp} = Z$. Hence, by Proposition~\ref{prop2}, $(X,B_{Y},\mc{K}(X))$ has property-$(P_{1})$. Therefore, the result follows from Proposition~\ref{prop3}. 
				\end{proof}	
				
				We now prove the main result of this section.
				
				\begin{thm}\label{thm5}
					Let $Y$ be a finite co-dimensional subspace of an $L_1$-predual space~$X$. Then the following are equivalent:
					\blr
					\item $Y$ is strongly proximinal in $X$.
					\item $Y$ is strongly ball proximinal in $X$.
					\item $(X,Y,\mc{K}(X))$ has property-$(P_{1})$.
					\item $(X,B_{Y},\mc{K}(X))$ has property-$(P_{1})$.
					\item $Y^\perp \ci \{x^\ast \in X^\ast: x^\ast\mbox{ is an SSD-point of }X^\ast\}$.
					\el
				\end{thm}
				\begin{proof}
					By \cite[Theorem~2.6]{C}, $(i) \Leftrightarrow (ii) \Leftrightarrow (v)$. Obviously, $(iii) \Rightarrow (i)$ and from Proposition~\ref{prop0.2}, $(iv) \Rightarrow (iii)$.
					
					Now, we prove that $(ii) \Rightarrow (iv)$. Assume $Y$ is strongly ball proximinal in $X$. Since $(ii) \Rightarrow (i)$, by \cite[Theorem~3.10]{CT}, $Y^{\perp \perp}$ is strongly proximinal in $X^{\ast \ast}$. Now, by \cite[Theorem~6.1]{JL}, $X^{\ast \ast}$ is isometric to $C(S)$, for some compact Hausdorff space~$S$. It follows from \cite[Theorem~2.1]{SD} and Theorem~\ref{thm2.1} that $(X^{\ast \ast},B_{Y^{\perp \perp}},\mc{K}(X^{\ast \ast}))$ has property-$(P_{1})$. Then, by Proposition~\ref{prop3}, $(X,B_{Y},\mc{K}(X))$ has property-$(P_{1})$.
				\end{proof}	
				
				We conclude this section by presenting characterizations for a strongly proximinal finite co-dimensional subspace of an $L_{1}$-predual space which are similar and in addition to those stated in \cite[Corollary~2.7]{C}.
				\begin{cor}\label{cor2}
					Let $Y$ be a finite co-dimensional subspace of an $L_1$-predual space~$X$. Then the following statements are equivalent:
					\blr
					\item $(X,Y,\mc{K}(X))$ has property-$(P_{1})$.
					\item $(X,B_{Y},\mc{K}(X))$ has property-$(P_{1})$.
					\item $Y$ is the intersection of finitely many hyperplanes $Y_{1},\ldots,Y_{n}$ such that for each $i=1,\ldots,n$, $(X,Y_{i},\mc{K}(X))$ has property-$(P_{1})$.  
					\item $Y$ is the intersection of finitely many hyperplanes $Y_{1},\ldots,Y_{n}$ such that for each $i=1,\ldots,n$, $(X,B_{Y_{i}},\mc{K}(X))$ has property-$(P_{1})$.  
					\el
				\end{cor}
				\begin{proof}
					It follows from Theorem~\ref{thm5} and \cite[Corollary~3.21]{CT} that $(i) \Leftrightarrow (iii)$ and $(ii) \Leftrightarrow (iv)$. Clearly, $(i) \Leftrightarrow (ii)$ follows from Theorem~\ref{thm5}.
				\end{proof}
				
				
				\section{An example of a subspace which satisfies $1\frac{1}{2}$-ball property and does not have \trcp}\label{sec5}
				A. L. Garkavi presented an example in \cite{GA} of a hyperplane in a non-reflexive Banach space which is proximinal but does not admit restricted Chebyshev center for a two-point set after a renorming. It can be observed that this hyperplane satisfies $1 \frac{1}{2}$-ball property in the renormed Banach space. This in turn shows that $1 \frac{1}{2}$-ball property and hence, strong proximinality is not a sufficient condition for \ercp. We now briefly describe Garkavi's example and prove that it satisfies $1 \frac{1}{2}$-ball property for the sake of completeness.
				
				\begin{ex}\label{ex0}
					Let $X$ be a non-reflexive Banach space and $Y=ker(x^\ast)$, where $x^{\ast} \in X^{\ast} \backslash \{0\}$, be a closed hyperplane in $X$. Then $Y$ is also non-reflexive and by James' theorem, there exists a linear functional $\Phi \in Y^{\ast}$ such that $\|\Phi\| =1$ and $\Phi$ does not attain its norm on $B_{Y}$. Define $D = \{y \in B_{Y}: \Phi(y) \geq \frac{3}{4}\}$ and then choose a $\ga >0$ and $y_{0} \in D$ such that $B[y_{0}, \ga] \cap Y$ is contained in the interior of the set $D$, w.r.t. $Y$. Let $\al = \inf\{\Phi(y):y \in B[y_{0},\ga] \cap Y\}$. Then $\frac{3}{4} \leq \al < 1$. Further, let us define $U = \{y \in B_{Y}: |\Phi(y)| \leq \al\}.$ Now, $U \cap B[y_{0},\ga] \cap Y = \es$ because the infimum defining $\al$ is not attained on $B[y_{0},\ga] \cap Y$.
					
					Let us fix $x_{0} \in X \backslash Y$ such that $x^{\ast}(x_0)=1$. We define $B_{\ga} = B[0,\ga] \cap Y$ and $V=  x_0 + B_{\ga}$. Let $B$ denote the closure of the set $conv(U \cup V \cup -V)$. Then $B$ is a closed bounded symmetric subset of $X$. Let $X^{\prime}$ denote the Banach space $X$, renormed to have $B$ as the closed unit ball. Let the renorming be denoted by $\|.\|_{B}$. Then the new norm $\|.\|_{B}$ on $X^{\prime}$ is equivalent to the old one on $X$. It is proved in \cite{GA} that $Y$ is proximinal in $X^{\prime}$ and $\textrm{cent}_{Y}(\{0,x_{0}+y_{0}\}) = \es$ in $X^{\prime}$.
				\end{ex}
				
				Let $Y$ be a subspace of a Banach space $X$. For an element $x \in X$ and $\e=0$, we note here that $P_Y(x,\e) = P_Y(x)$. For a non-empty subset $A$ of $X$ and $x \in X$, we denote $x+A = \{x+a : a \in A\}$. Let us now recall a characterisation of $1 \frac{1}{2}$-ball property provided in \cite{Godini}. The following result follows directly from \cite[Remark~6, p.~50 and Corollary~4, p.~52]{Godini}. 
				\begin{prop}\label{P1}
					Let $Y$ be a subspace of a Banach space $X$. Then $Y$ has $1 \frac{1}{2}$-ball property in $X$ if and only if $Y$ is proximinal in $X$ and for each $x \in X$ and $\e\geq 0$, $P_{Y}(x,\e) = \{y \in Y: d(y,P_Y(x)) \leq \e\}$.
				\end{prop}
				
				The proof idea for the following result is similar to that used in \cite[Example~3.3]{BLR}.
				\begin{prop}\label{prop}
					Let $Y$ be a closed hyperplane in a non-reflexive Banach space $X$ and $X^{\prime}$ be the Banach space $X$ with the renorming $\|.\|_{B}$ as defined in the Example~\ref{ex0}. Then $Y$ satisfies $1 \frac{1}{2}$-ball property in $X^{\prime}$.
				\end{prop}
				\begin{proof}
					Clearly, if $x \in X^\prime$, then there exists $\la \in \mb{R}$ and $y \in Y$ such that $x = y + \la x_0$. Also, clearly, $P_{Y}(y+\la x_0) = y+ \la P_{Y}(x_0)$ and $P_{Y}\left(y+\la x_0, \de\right) = y+ \la P_{Y}(x_0, \frac{\de}{|\la|})$, for $\de >0$ and $\la \neq 0$. Therefore, applying Proposition~\ref{P1} and by translation, it suffices to prove that for each $\e\geq 0$, $P_{Y}(x_0,\e) = \{y \in Y: d(y,P_Y(x_0)) \leq \e\}$. Now, $d(x_0, Y) = 1$ and $P_{Y}(x_0) = B_{\ga}$. Let $\e\geq0$. By \cite[Remark~5, p.~50]{Godini}, we have $\{y \in Y: d(y,P_Y(x_0)) \leq \e\} \ci P_{Y}(x_0, \e)$. For $\e=0$, it is trivial to see that $P_Y(x_0) \ci \{y \in Y: d(y,P_{Y}(x_0))=0\}$. Thus, it remains to show that for each $\e>0$, $P_{Y}(x_0, \e) \ci \{y \in Y: d(y,P_Y(x_0)) \leq \e\}$, or in other words, we prove that if $\e>0$ and $y \in Y$ is such that $\|y-x_0\|_{B} \leq 1 + \e$, then we have $d(y, B_{\ga})\leq\e$.
					
					Let $y \in Y$ such that $\eta = \|y-x_0\|_{B} \leq 1 + \e$. Without loss of generality, assume $\eta>1$. Therefore, $\frac{y-x_0}{\eta} \in B$. Thus, there exists sequences $\{\al_n\}$, $\{\be_n\}$, $\{\nu_n\}$ $\ci [0,1]$ such that for each $n$, $\al_n + \be_n + \nu_n = 1$ and sequences $\{u_{n}\}$, $\{u_{n}^{\prime}\}$ $\ci B_{\ga}$; $\{y_{n}\} \ci U$ such that \[\frac{y-x_0}{\eta}= \lim_{n \rightarrow \iy} [\al_n u_n + \be_n u_n^{\prime} + \nu_n y_n + (\al_n -\be_n) x_0].\] Without loss of generality, assume $\al_n \rightarrow \al$, $\be_n \rightarrow \be$ and $\nu_n \rightarrow \nu$, where $\al,\be,\nu \in [0,1]$ and $\al+\be +\nu =1$. Therefore, it follows that $\be - \al = \frac{1}{\eta}$ and $y = \lim_{n \rightarrow \iy} \eta[\al u_n + \be u_n^{\prime} + \nu y_n]$.
					Now, $\frac{1}{\eta} \leq \al + \frac{1}{\eta} = \be \leq 1$ and for each $n$, $\|u_n\|_{B}$, $\|u_{n}^{\prime}\|_{B}$, $\|y_n\|_{B} \leq 1$. Therefore, 
					\begin{equation}
						\begin{split}
							d(y,B_{\ga}) &\leq \inf_{n} \|y - u_{n}^{\prime}\|_{B}\\ &\leq \liminf_{n} \|\eta [\al u_{n} + \be u_{n}^{\prime} + \nu y_{n}] -u_{n}^{\prime}\|_{B}\\ &=\liminf_{n} \| \eta \al u_{n} + (\eta \be-1)u_{n}^{\prime} + \eta \nu y_{n}\|_{B} \\ &\leq \eta \al + (\eta \be-1) + \eta \nu= \eta-1 \leq \e.\\
						\end{split}
					\end{equation}	
				\end{proof}	
				\bibliographystyle{plain}
				\bibliography{Bibliography_property(P_1)}
			
			\end{document}